\documentclass{elsarticle}
\usepackage{lineno,hyperref}
\usepackage{amssymb,amsmath,amsthm} 
\theoremstyle{plain}
 \newtheorem{theorem}{Theorem}[section]
 
 \newtheorem{lem}{Lemma}[section]
 \newtheorem{prop}{Proposition}[section]
  
\theoremstyle{definition}

\theoremstyle{remark}
 \newtheorem{rem}{Remark}[section] 
\begin{document}
\begin{frontmatter}
\title{{Positive solutions for singular\\  double phase problems}}
\author[Nikolaos S. Papageorgiou]{{Nikolaos S. Papageorgiou}}
\ead{npapg@math.ntua.gr}
\address[Nikolaos S. Papageorgiou]{National Technical University, Zografou campus, 15780, Athens, Greece  }
\author[Dusan D. Repovs]{{Du\v{s}an D. Repov\v{s}}\corref{mycorrespondingauthor}}\cortext[mycorrespondingauthor]{Corresponding author}
\ead{dusan.repovs@guest.arnes.si}
\address[Dusan D. Repovs]{Faculty of Education and Faculty of Mathematics and Physics,
	University of Ljubljana \\ \& Institute of Mathematics, Physics and Mechanics, SI-1000,
	Ljubljana, Slovenia}
\author[Calogero Vetro]{{Calogero Vetro}}
\ead{calogero.vetro@unipa.it}
\address[Calogero Vetro]{Department of Mathematics and Computer Science, University of Palermo,  90123, Palermo, Italy}
\begin{abstract}
	We study the existence of positive solutions for a class of double phase Dirichlet equations which have the combined effects of a singular term and of a parametric superlinear term. The differential operator of the equation is the sum of a $p$-Laplacian and of a weighted $q$-Laplacian ($q<p$) with discontinuous weight. Using the Nehari method, we show that for all small values of the parameter $\lambda>0$, the equation has at least two positive solutions. 
\end{abstract}
\begin{keyword}
Double phase problem \sep singular term \sep Nehari manifold \sep positive solutions \sep discontinuous weight
\MSC[2010] 35J60 \sep 35D05
\end{keyword}
\end{frontmatter}

	\section{Introduction}
Let $\Omega \subseteq \mathbb{R}^N$ be a bounded domain with Lipschitz boundary $\partial \Omega$. In this paper we study the following singular double phase problem
\begin{equation}\tag{$P_\lambda$}\label{PL} 
\begin{cases}
- \Delta_p u - \mbox{div}\, \left( \xi(z)|\nabla u|^{q-2}\nabla u \right) = a(z)u^{-\gamma} +\lambda u^{r-1}  \mbox{ in }\Omega,& \\
\quad u \big|_{\partial \Omega}=0, \, 1<q<p<r<p^*, \, 0<\gamma <1, \, u \geq 0, \, \lambda >0.
\end{cases}
\end{equation} 

Here,
 $\Delta_p$ denotes the $p$-Laplace differential operator defined by
$$\Delta_p=\mbox{div}\, \left( |\nabla u|^{p-2}\nabla u \right) \quad \mbox{for all } u \in W_0^{1,p}(\Omega).$$

The weight $\xi: \Omega \to \mathbb{R}_+$ is essentially bounded. Thus the differential operator in \eqref{PL} is the sum of a $p$-Laplacian and of a weighted $q$-Laplacian ($q<p$). The integrand in the energy functional of this operator is
$$k(z,t)=\frac{1}{p}t^p + \frac{1}{q}\xi(z)t^q \quad \mbox{for all }t>0.$$

This is a Carath\'{e}odory function (that is, for all $t>0$, $z \to k(z,t)$ is measurable, and for a.a. $z \in \Omega$, $t \to k(z,t)$ is continuous) which exhibits balanced growth in the $t>0$ variable, that is,
$$\frac{1}{p}t^p \leq k(z,t) \leq c_0[1+t^p] \quad \mbox{for a.a. $z \in \Omega$, all }t>0.$$

However, the presence of the weight $\xi(\cdot)$, which is discontinuous and  not bounded away from zero, does not permit the use of the global regularity theory of Lieberman \cite{Ref9} and of the nonlinear strong maximum principle of Pucci-Serrin \cite{Ref15} (p. 111, 120). The absence of these basic tools leads to a different approach based on the Nehari method. In the reaction (the right hand side), we have the combined effects of a singular term and of a parametric $(p-1)$-superlinear perturbation. We are looking for positive solutions and we show that for all small values $\lambda>0$ of the parameter, problem \eqref{PL} has at least two positive solutions.

Double phase equations have been studied by Cencelj-R\v{a}dulescu-Repov\v{s} \cite{Ref3} (problems with variable growth), Colasuonno-Squassina \cite{Ref4}, Colombo-Mingione \cite{Ref5,Ref6}, Baroni-Colombo-Mingione \cite{Ref1}, and Liu-Dai \cite{Ref10} (problems with a differential operator which exhibits unbalanced growth). A nice survey of the recent works on such equations can be found in R\v{a}dulescu \cite{Ref16}. We also mention the recent works on $(p,q)$-equations (equations driven by the sum of a $p$-Laplacian and of a $q$-Laplacian) with singular terms of Papageorgiou-R\v{a}dulescu-Repov\v{s} \cite{Ref12} and Papageorgiou-Vetro-Vetro \cite{Ref13}. For such differential operators, the integrand of the energy functional is
$k(t)=\frac{1}{p}t^p + \frac{1}{q}t^q$ for all $t>0$ (that is, $\xi(z)=1$) and so the use of the global regularity theory of Lieberman \cite{Ref9} and  the nonlinear maximum principle of Pucci-Serrin \cite{Ref15} is possible. This fact in turn, permits the use of truncation and comparison techniques, which make it
possible to bypass the singularity in the reaction.

The main result of our paper is the following multiplicity theorem for problem \eqref{PL}.

\begin{theorem}\label{MT}
	\label{P8} If hypotheses $H(\xi)$, $H(a)$ hold, then there exists $\widehat{\lambda}^\ast_0>0$ such that for all $\lambda \in (0, \widehat{\lambda}^\ast_0]$  problem \eqref{PL} has at least two positive solutions $u^\ast,v^\ast \in W^{1,p}_0(\Omega)$ such that $\varphi_\lambda(u^\ast)<0\leq \varphi_\lambda(v^\ast)$.
\end{theorem}

\section{Preliminaries}\label{S2}

By $W_0^{1,p}(\Omega)$ we denote the usual ``Dirichlet'' Sobolev space and by $\|\cdot \|$ we denote the norm of $W_0^{1,p}(\Omega)$. The Poincar\'e inequality (see Papageorgiou-R\v{a}dulescu-Repov\v{s} \cite{Ref11}, p. 43) implies that we can have
$$\|u\|=\|\nabla u\|_p \quad \mbox{for all }W_0^{1,p}(\Omega).$$

Here, by $\|\cdot \|_s$ ($1 \leq s \leq +\infty$) we denote the norm of $L^{s}(\Omega,\mathbb{R}^m)$, $m \in \mathbb{N}$. Also, by $|\cdot |$ we denote the norm of $\mathbb{R}^N$ and by $p^\ast$ the critical Sobolev exponent corresponding to $p$, that is 
$$p^\ast=\begin{cases}\dfrac{Np}{N-p} & \mbox{if } p<N,\\ +\infty & \mbox{if } N \leq p.\end{cases}$$

The hypotheses on the data of \eqref{PL} are the following:
\begin{itemize}
	\item[$H(\xi)$:]  $\xi \in L^\infty(\Omega)$ and $\xi(z) > 0$ for a.a. $z \in \Omega$.
	\item[$H(a)$:]  $a \in L^\infty(\Omega)$ and $a(z) \geq 0$ for a.a. $z \in \Omega$, $a \not \equiv 0$.
\end{itemize}

The energy (Euler) functional for this problem $\varphi_\lambda : W_0^{1,p}(\Omega) \to \mathbb{R}$ is given by
\begin{align*}\varphi_\lambda(u)&=\frac{1}{p}\|\nabla u\|^p_p +\frac{1}{q} \int_\Omega \xi(z)|\nabla u|^q dz - \frac{1}{1-\gamma} \int_{\Omega} a(z)|u|^{1-\gamma} dz - \frac{\lambda}{r}\|u\|_r^r\\ &  \hskip 7cm \mbox{for all } u \in W_0^{1,p}(\Omega).\end{align*}

On account of the singular term $a(z)u^{-\gamma}$, this functional is not $C^1$. So, the use of variational methods based on the critical point theory presents difficulties which are compounded by the fact that the weight $\xi(\cdot)$ is discontinuous and not bounded away from zero. For this reason our approach is based on the Nehari method.

Recall that $u \in W_0^{1,p}(\Omega)$ is a weak solution of \eqref{PL}, if $u(z)\geq 0$ for a.a. $z \in \Omega$, $u \not \equiv 0$ and 
\begin{align*} & \int_\Omega |\nabla u|^{p-2} (\nabla u, \nabla h)_{\mathbb{R}^N}dz + \int_\Omega \xi(z)|\nabla u|^{q-2} (\nabla u, \nabla h)_{\mathbb{R}^N}dz\\ & = \int_\Omega a(z)u^{-\gamma}h dz + \lambda \int_\Omega u^{r-1}h dz \quad \mbox{for all }h \in W_0^{1,p}(\Omega).
\end{align*}	

For every $\lambda>0$, we introduce the Nehari manifold for problem \eqref{PL} defined by
{\small 
$$N_\lambda =\left\{u \in W_0^{1,p}(\Omega) \, : \, \|\nabla u\|_p^p + \int_\Omega \xi(z) |\nabla u |^q dz = \int_\Omega a(z) |u|^{1-\gamma}dz + \lambda \|u\|_r^r, \, u \neq 0\right\}.$$
}
Evidently,
 the Nehari manifold contains the weak solutions of \eqref{PL} and as we will see in the sequel, for small $\lambda >0$ one has $N_\lambda \neq \emptyset$. The Nehari manifold is much smaller than $W_0^{1,p}(\Omega)$ and so $\varphi_\lambda \Big|_{N_\lambda}$ can have nice properties which fail to be true globally.

It will be helpful to decompose $N_\lambda$ into three disjoint parts:
 
\begin{align*}
	&N_\lambda^+ =\Big\{u \in N_\lambda \, : \,(p+\gamma-1) \|\nabla u\|_p^p +(q+\gamma-1) \int_\Omega \xi(z) |\nabla u |^q dz   \\ & \hskip 8cm   -\lambda (r+\gamma-1) \|u\|_r^r> 0\Big\},\\
	& N^0_\lambda =\Big\{u \in N_\lambda \, : \,(p+\gamma-1) \|\nabla u\|_p^p +(q+\gamma-1) \int_\Omega \xi(z) |\nabla u |^q dz \\ & \hskip 8cm =\lambda (r+\gamma-1) \|u\|_r^r \Big\},\\
	&N^-_\lambda =\{u \in N_\lambda \, : \,(p+\gamma-1) \|\nabla u\|_p^p +(q+\gamma-1) \int_\Omega \xi(z) |\nabla u |^q dz \\ & \hskip 8cm -\lambda (r+\gamma-1) \|u\|_r^r< 0\Big\}.
\end{align*}	 

\section{The proof of Theorem~\ref{MT}}

In this section, using the Nehari method, we   shall prove our main result, Theorem~\ref{MT}, which asserts that for all small $\lambda >0$, problem \eqref{PL} has at least two positive solutions. Our proof will be broken down in a sequence of propositions.

\begin{prop}
	\label{P1} If hypotheses $H(\xi)$, $H(a)$ hold and $\lambda >0$, then $\varphi_\lambda \Big|_{N_\lambda}$ is coercive.
\end{prop}
\begin{proof}
	Let $u \in N_\lambda$. From the definition of $N_\lambda$, we have 
	\begin{equation}
	\label{eq1} - \frac{1}{r}\|\nabla u\|_p^p - \frac{1}{r} \int_\Omega \xi(z) |\nabla u |^q dz +\frac{1}{r} \int_\Omega a(z) |u|^{1-\gamma}dz + \frac{\lambda}{r} \|u\|_r^r= 0.
	\end{equation}	
	
	Using \eqref{eq1}, we have
	\begin{align*}
		&	\varphi_\lambda(u)=\left[\frac{1}{p} -\frac{1}{r}\right]\|\nabla u\|^p_p +\left[\frac{1}{q} -\frac{1}{r}\right] \int_\Omega \xi(z)|\nabla u|^q dz\\ & \hskip 1,5cm + \left[\frac{1}{r} -\frac{1}{1-\gamma}\right]  \int_{\Omega} a(z)|u|^{1-\gamma} dz\\ \Rightarrow \quad &  	\varphi_\lambda(u) \geq  c_1\|u\|^p - c_2 \|u\|^{1-\gamma} \quad \mbox{for some $c_1,c_2>0$ (since $q<p<r$)}.
	\end{align*}
	
	Here we have used
	the Poincar\'e's inequality, Theorem 13.17
	on
	 p.196 of Hewitt-Stromberg \cite{Ref8} and the Sobolev embedding theorem. From the last inequality and since $p>1>1-\gamma$, we
	 can
	  conclude that $\varphi_\lambda \Big|_{N_\lambda}$ is coercive.	
\end{proof}

Let $m_\lambda^+=\inf_{N_\lambda^+}\varphi_\lambda$.

\begin{prop}
	\label{P2} If hypotheses $H(\xi)$, $H(a)$ hold and $N_\lambda^+ \neq \emptyset$, then $m_\lambda^+<0$.
\end{prop}
\begin{proof}
	By the definition of $N_\lambda^+$, we have 
	\begin{equation}\label{eq2}
	\lambda \|u\|_r^r < \frac{p+\gamma-1}{r+\gamma-1} \|\nabla u\|_p^p  +\frac{q+\gamma-1}{r+\gamma-1} \int_\Omega \xi(z)|\nabla u|^q dz \quad \mbox{for all }  u \in N_\lambda^+.
	\end{equation}	
	
	We know that $N_\lambda^+ \subseteq N_\lambda$. So, we have 	
	
	\begin{equation}
	\label{eq3} -\frac{1}{1-\gamma}\int_\Omega a(z)|u|^{1-\gamma}dz = - \frac{1}{1-\gamma}\|\nabla u\|_p^p - \frac{1}{1-\gamma} \int_\Omega \xi(z)|\nabla u|^q dz + \frac{\lambda}{1-\gamma} \|u\|^r_r
	\end{equation}
	
	for all $u \in N_\lambda^+$. Now, 	for all $u \in N_\lambda^+$, we have 
	
	\begin{align*}
	&	\varphi_\lambda(u)\\& =\frac{1}{p}\|\nabla u\|^p_p +\frac{1}{q} \int_\Omega \xi(z)|\nabla u|^q dz - \frac{1}{1-\gamma} \int_{\Omega} a(z)|u|^{1-\gamma} dz - \frac{\lambda}{r}\|u\|_r^r \\ & =\left[\frac{1}{p} -\frac{1}{1-\gamma}\right]\|\nabla u\|^p_p +\left[\frac{1}{q} -\frac{1}{1-\gamma}\right] \int_\Omega \xi(z)|\nabla u|^q dz +\lambda  \left[\frac{1}{1-\gamma} -\frac{1}{r}\right]  \|u\|_r^r \\ & \hskip 10cm \mbox{(see \eqref{eq3})} \\	& \leq \left[ \frac{-(p+\gamma-1)}{p(1-\gamma)}+\frac{p+\gamma-1}{r+\gamma-1}\frac{r+\gamma-1}{r(1-\gamma)}  \right]\| \nabla u\|^p_p\\ & \quad +  \left[ \frac{-(q+\gamma-1)}{q(1-\gamma)}+\frac{q+\gamma-1}{r+\gamma-1}\frac{r+\gamma-1}{r(1-\gamma)}  \right] \int_\Omega \xi(z) |\nabla u |^q dz \quad \mbox{(see \eqref{eq2})}\\ & = \frac{p+\gamma-1}{1-\gamma}\left[\frac{1}{r} -\frac{1}{p}\right]\|\nabla u\|_p^p + \frac{q+\gamma-1}{1-\gamma}\left[\frac{1}{r} -\frac{1}{q}\right]\int_\Omega \xi(z) |\nabla u |^q dz\\ & <0 \quad \mbox{(see hypothesis $H(\xi)$ and recall that $q<p<r$),}\\ \Rightarrow \, & \varphi_\lambda \Big|_{N_\lambda^+}<0,\\  \Rightarrow \, & m_\lambda^+<0.
	\end{align*}		
\end{proof}	

\begin{prop}
	\label{P3} If hypotheses $H(\xi)$, $H(a)$ hold, then there exists $\lambda^\ast >0$ such that $N_\lambda^0 = \emptyset$ for all $\lambda \in (0,\lambda^\ast)$.
\end{prop}
\begin{proof}
	Arguing by contradiction, suppose that $N_\lambda^0 \neq \emptyset$ for all $\lambda >0$. So, for every $\lambda >0$, we can find $u \in N_\lambda$ such that
	\begin{equation}
	\label{eq4} (p+\gamma-1)\|\nabla u\|_p^p + (q+\gamma-1)\int_\Omega \xi(z)|\nabla u|^q dz = \lambda (r+\gamma -1)\|u\|_r^r.
	\end{equation}
	
	Since $u \in N_\lambda$, we also have
	\begin{align}
\nonumber & (r+\gamma-1)\|\nabla u\|_p^p + (r+\gamma-1)\int_\Omega \xi(z)|\nabla u|^q dz - (r+\gamma-1)\int_\Omega a(z)| u|^{1-\gamma} dz \\ & \hskip 7cm= \lambda (r+\gamma -1)\|u\|_r^r.	\label{eq5}
	\end{align}
	
	We subtract \eqref{eq4} from \eqref{eq5} and obtain
	\begin{align}
		\nonumber &(r-p)\|\nabla u\|_p^p + (r-q) \int_\Omega \xi(z)|\nabla u|^q dz = (r+\gamma -1) \int_\Omega a(z)|u|^{1-\gamma}dz,\\  \nonumber
		\Rightarrow \quad & \|u\|^p \leq c_3 \|u\| \quad \mbox{for some }c_3 >0,\\ 
		\Rightarrow \quad & \|u\|^{p-1} \leq c_3.
		\label{eq6}	\end{align}
	
	From \eqref{eq4} and the Sobolev embedding theorem, we have
	\begin{align*}
		&	\|u\|^p \leq \lambda c_4 \|u\|^r \quad \mbox{ for some } c_4 >0,\\ \Rightarrow \quad & \left[\dfrac{1}{\lambda c_4}\right]^{\frac{1}{r-p}}\leq \|u\|.
	\end{align*}
	
	If $\lambda \to 0^+$, then $\|u\| \to +\infty$ and this contradicts \eqref{eq6}. This shows that there exists $\lambda^\ast >0$ such that $N_\lambda^0 = \emptyset$ for all $\lambda \in (0,\lambda^\ast)$.
\end{proof}

Now let $u \in W_0^{1,p}(\Omega)$ and consider the function $\widehat{w}_u: (0,+\infty) \to \mathbb{R}$ defined by
$$\widehat{w}_u(t)=t^{p-r} \|\nabla u \|_p^p - t^{-r-\gamma+1}\int_\Omega a(z)|u|^{1-\gamma}dz \quad \mbox{for all }t>0.$$

Since $r-p<r+\gamma -1$, we see that there exists $\widehat{t}_0>0$ such that 
$$\widehat{w}_u(\widehat{t}_0)=\max_{t>0}\widehat{w}_u.$$

Then we have
\begin{align*} & \widehat{w}_u^\prime(\widehat{t}_0)=0, \\ \Rightarrow \quad & (p-r)\widehat{t}_0^{p-r-1}\|\nabla u\|_p^p + (r+\gamma -1)\widehat{t}_0^{-r-\gamma}\int_\Omega a(z)|u|^{1-\gamma}dz=0, \\ \Rightarrow \quad 
	& \widehat{t}_0= \left[\frac{(r +\gamma -1)\int_\Omega a(z)|u|^{1-\gamma}dz}{(r-p)\| \nabla u\|_p^p}\right]^{\frac{1}{p+\gamma-1}}.
\end{align*}
Therefore we have

\begin{align} \nonumber 
	\widehat{w}_u(\widehat{t}_0)=	&  \frac{\left[(r-p)\| \nabla u\|_p^p\right]^{\frac{r-p}{p+\gamma-1}}}{\left[(r +\gamma -1)\int_\Omega a(z)|u|^{1-\gamma}dz\right]^{\frac{r-p}{p+\gamma-1}}} \|\nabla u\|^p_p\\ \nonumber & \quad -	\frac{\left[(r-p)\| \nabla u\|_p^p\right]^{\frac{r+\gamma-1}{p+\gamma-1}}}{\left[(r +\gamma -1)\int_\Omega a(z)|u|^{1-\gamma}dz\right]^{\frac{r+\gamma-1}{p+\gamma-1}}}\int_\Omega a(z)|u|^{1-\gamma}dz\\ \nonumber &= \frac{(r-p)^{\frac{r-p}{p+\gamma-1}}\| \nabla u\|_p^{\frac{p(r+\gamma-1)}{p+\gamma-1}}}{\left[(r +\gamma -1)\int_\Omega a(z)|u|^{1-\gamma}dz\right]^{\frac{r-p}{p+\gamma-1}}} \\ \nonumber & \quad - \frac{(r-p)^{\frac{r+\gamma-1}{p+\gamma-1}}\| \nabla u\|_p^{\frac{p(r+\gamma-1)}{p+\gamma-1}}}{\left[(r +\gamma -1)\int_\Omega a(z)|u|^{1-\gamma}dz\right]^{\frac{r-p}{p+\gamma-1}}} \\ \label{eq7} &=\frac{p+\gamma-1}{r-p} \left[\frac{r-p}{r+\gamma-1}\right]^{\frac{r+\gamma-1}{p+\gamma-1}} \frac{\| \nabla u\|_p^{\frac{p(r+\gamma-1)}{p+\gamma-1}}}{\left[ \int_\Omega a(z)|u|^{1-\gamma}dz\right]^{\frac{r-p}{p+\gamma-1}}}. 
\end{align}

If $S$ denotes the best Sobolev constant, we have
\begin{equation}
\label{eq8}S \|u\|_{p^\ast}^p \leq \|\nabla u\|_p^p.
\end{equation}

Also, we have 
\begin{equation}
\label{eq9} \int_\Omega a(z)|u|^{1-\gamma}dz \leq c_5 \|u\|_{p^\ast}^{1-\gamma} \quad \mbox{for some $c_5>0$.}
\end{equation}

Then we have
\begin{align*} 
	&	\widehat{w}_u(\widehat{t}_0)-\lambda \|u\|^r_r	
	\\  &\geq \frac{p+\gamma-1}{r-p} \left[\frac{r-p}{r+\gamma-1}\right]^{\frac{r+\gamma-1}{p+\gamma-1}} \frac{S^{\frac{p(r+\gamma-1)}{p+\gamma-1}}\left(\|   u\|_{p^\ast}^p\right)^{\frac{r+\gamma-1}{p+\gamma-1}}}{\left(c_5 \|u\|_{p^\ast}^{1-\gamma}\right)^{\frac{r-p}{p+\gamma-1}}}- \lambda c_6 \|u\|_{p^\ast}^{r}\\ & \hskip 3cm \mbox{for some $c_6>0$ (see \eqref{eq7}, \eqref{eq8}, \eqref{eq9} and recall $r<p^\ast$)}\\ & = \left[c_7-\lambda c_6\right]\|u\|_{p^\ast}^{r} \quad \mbox{for some $c_7>0$.}
\end{align*}

So, there exists $\widehat{\lambda}^\ast \in (0,\lambda^\ast]$ independent of $u$, such that
\begin{equation}
\label{eq10}\widehat{w}_u(\widehat{t}_0)-\lambda \|u\|^r_r>0 \quad \mbox{for all }\lambda \in (0, \widehat{\lambda}^\ast).
\end{equation}

\begin{prop}
	\label{P4} If hypotheses $H(\xi)$, $H(a)$ hold, then there exists $\widehat{\lambda}^\ast \in (0,\lambda^\ast]$ such that   for every $\lambda \in (0,\lambda^\ast)$ we can find $u^\ast \in N^+_\lambda$ such that $\varphi_\lambda(u^\ast)=m_\lambda^+<0$ and $u^\ast(z)\geq 0$ for a.a. $z \in \Omega$.
\end{prop}
\begin{proof}
	For $u \in W^{1,p}_0(\Omega)$ we consider the function $w_u : (0,+\infty) \to \mathbb{R}$ defined by
	$$w_u(t)=t^{p-r}\|\nabla u\|_p^p +t^{q-r}\int_\Omega \xi (z)|\nabla u|^q dz - t^{-r-\gamma+1}\int_\Omega a(z)|u|^{1-\gamma}dz  \mbox{ for all $t>0$.}$$
	
	Since $r-p<r-q<r+\gamma -1$, we can find $t_0>0$ such that 
	$$w_u(t_0)=\max_{t>0} w_u.$$
	
	Evidently, we have $w_u \geq \widehat{w}_u$ and so from \eqref{eq10} we see that we can find $\widehat{\lambda}^\ast \in (0,\lambda^\ast]$ such that
	$$w_u(t_0)-\lambda \|u\|^r_r>0 \quad \mbox{for all }\lambda \in (0, \widehat{\lambda}^\ast).$$
	
	Consequently, we can find $t_1<t_0<t_2$ such that
	\begin{equation}
	w_u(t_1)=\lambda \|u\|^r_r = w_u(t_2) \quad \mbox{and} \quad w^\prime_u(t_2)<0< w^\prime_u(t_1). \label{eq11}
	\end{equation}
	
	Now we see that
	$$t_1u \in N_\lambda^+ \quad \mbox{and} \quad t_2u \in N_\lambda^-.$$	
	
	Therefore for all $\lambda \in (0, \widehat{\lambda}^\ast)$, we have $N^\pm_\lambda \neq \emptyset$ while $N^0_\lambda =\emptyset$ (see Proposition \ref{P3}).
	
	Now consider a minimizing sequence $\{u_n\}_{n \geq 1} \subseteq N_\lambda^+$, that is, 
	$$\varphi_\lambda(u_n) \downarrow m_\lambda^+ \quad \mbox{as }n \to +\infty.$$
	
	On account of Proposition \ref{P1}, we have that
	$$\{u_n\}_{n \geq 1} \subseteq W_0^{1,p}(\Omega) \mbox{ is bounded (recall that $N_\lambda^+ \subseteq N_\lambda$)}.$$
	
	So, by passing to a suitable subsequence if necessary, we may assume that
	\begin{equation*}
	\label{eq12} u_n \xrightarrow{w} u^\ast \mbox{ in $W_0^{1,p}(\Omega)$ and $u_n \to u^\ast$ in $L^r(\Omega)$}. 
	\end{equation*}
	
	We consider the function $w_{u^\ast}(\cdot)$ and let $t_1<t_0$ be as in \eqref{eq11} (with $u=u^\ast$). From the first part of the proof we know that $t_1 u^\ast \in N_\lambda^+$.
	
	We claim that $u_n \to u^\ast$ in $W_0^{1,p}(\Omega)$ as $n \to +\infty$. Arguing by contradiction, suppose that $u_n \not \to u^\ast$ in $W_0^{1,p}(\Omega)$. Then we will have	
	 
	\begin{equation}
	\label{eq13} 
\liminf _{n \to +\infty } \| \nabla u_{n}\|_{p}^{p} > \| \nabla u\|_{p}^{p}.
	\end{equation}	
	
	For $u \in W^{1,p}_0(\Omega)$ we consider the fibering function $\mu_u : (0,+\infty) \to \mathbb{R}$ defined by
	$$\mu_u(t)=\varphi_\lambda(tu) \quad \mbox{for all } t>0.$$
	
	Using (\ref{eq13})  
(see also  
\cite{Ref2, 
Ref14}), 
we have
		\begin{align}
		\nonumber & \liminf_{n \to +\infty} \mu_{u_n}^\prime(t_1)\\ 	
		\nonumber & = \liminf_{n \to +\infty} \left[t_1^{p-1}\|\nabla u_n\|_p^p + t_1^{q-1}\int_\Omega \xi(z)|\nabla u_n|^q dz -t_1^{-\gamma}\int_\Omega a(z)|u_n|^{1-\gamma}dz \right.\\ \nonumber &   \left. \hskip 2cm-\lambda t_1^{r-1} \|u_n\|^r_r\right]\\ 	
		 \nonumber & > t_1^{p-1}\|\nabla u^\ast\|_p^p + t_1^{q-1}\int_\Omega \xi(z)|\nabla u^\ast|^q dz -t_1^{-\gamma}\int_\Omega a(z)|u^\ast|^{1-\gamma}dz -\lambda t_1^{r-1} \|u^\ast\|^r_r \\ \nonumber &    \hskip 9cm \mbox{ (see \eqref{eq13})}\\ \label{eq14} & = \mu_{u^\ast}^\prime(t_1)=0\mbox{ (see \eqref{eq11})}.
	\end{align}		
	Then
	it follows 
	 from \eqref{eq14} 
	 that we can find $n_0 \in \mathbb{N}$ such that
	\begin{equation}
	\label{eq15}\mu_{u^\ast}^\prime(t_1)>0 \quad \mbox{for all $n \geq n_0$}.
	\end{equation}
	
	Since $u_n \in N_\lambda^+ \subseteq N_\lambda$ and $\mu_{u_n}^\prime(t)=t^r \left[w_{u_n}(t)-\lambda \|u_n\|_r^r\right]$, we have
	\begin{align*}
		& 	\mu_{u_n}^\prime(t)<0 \mbox{ for all $t \in (0,1)$ and $\mu_{u_n}^\prime(1)=0$,}\\ \Rightarrow \quad & t_1>0 \mbox{ (see \eqref{eq15}).}
	\end{align*}

	The function $\mu_{u^\ast}(\cdot)$ is decreasing on $(0,t_1)$. Hence we have
	\begin{equation}
	\label{eq16} \varphi_\lambda(t_1 u^\ast) \leq \varphi_\lambda(u^\ast)< m_\lambda^+ \quad \mbox{(see \eqref{eq13})}.
	\end{equation}
	
	However, $t_1u^\ast \in N_\lambda^+$. Hence 
	$$m_\lambda^+ \leq \varphi_\lambda(tu^\ast)< m_\lambda^+ \quad \mbox{(see \eqref{eq16})},$$
	a contradiction. This proves that our initial claim holds and we have
	\begin{align}
		&	u_n \to u^\ast \mbox{ in $W_0^{1,p}(\Omega)$}, \label{eq17} \\ \nonumber \Rightarrow \quad & \varphi_\lambda(u_n) \to \varphi_\lambda(u^\ast),\\ \nonumber \Rightarrow \quad & \varphi_\lambda(u^\ast) =m_\lambda^+.
	\end{align}
	
	Since $u_n \in N^+_\lambda$ for all $n \in \mathbb{N}$, we have 
	\begin{align}
		\nonumber & (p+\gamma -1) \|\nabla u_n\|_p^p +(q+\gamma -1) \int_{\Omega} \xi(z)|\nabla u_n |^q dz > \lambda (r+\gamma -1)\|u_n\|^r_r, \\ \label{eq18} \Rightarrow \, & (p+\gamma -1) \|\nabla u^\ast\|_p^p +(q+\gamma -1) \int_{\Omega} \xi(z)|\nabla u^\ast |^q dz \geq \lambda (r+\gamma -1)\|u^\ast\|^r_r \\ & \hskip 9cm \nonumber \mbox{(see \eqref{eq17})}.
	\end{align}
	
	However,
	 $\lambda \in (0, \widehat{\lambda}^\ast)$ and $\widehat{\lambda}^\ast \leq \lambda^\ast$. So, from Proposition \ref{P3}, we know that $N_\lambda^0 =\emptyset$. Therefore in \eqref{eq18} equality cannot hold and so we
	 can
	  conclude that $u^\ast \in N_\lambda^+$.		
	
	Clearly,
	 we can replace $u^\ast$ by $|u^\ast|$ and so we can say that $u^\ast(z) \geq 0$ for a.a. $z \in \Omega$. 
\end{proof}

We will need the following lemma which was inspired by Lemma 3 of Sun-Wu-Long \cite{Ref17}. In what follows,
 $B_\varepsilon(0)=\{w \in W^{1,p}_0(\Omega) \, : \, \|w\| <\varepsilon \}$, $\varepsilon >0$. 

\begin{lem}
	\label{L5} If hypotheses $H(\xi)$, $H(a)$ hold and $u \in N^\pm_\lambda$, then there exist $\varepsilon >0$ and a continuous function $\beta:B_\varepsilon(0) \to \mathbb{R}_+$ such that $\beta(0)=1$, $\beta(w)(u+w) \in N_\lambda^+$ for all $w \in B_\varepsilon(0)$.
\end{lem}
\begin{proof}
	We shall only give the proof for $u \in N_\lambda^+$, the proof for $u \in N_\lambda^-$ is similar.
	
	Consider the function $E: W_0^{1,p}(\Omega) \times \mathbb{R} \to \mathbb{R}$ defined by
	\begin{align*} & E(w,t)=t^{p+\gamma-1}\|\nabla (u+w)\|_p^p + t^{q+\gamma-1} \int_\Omega \xi(z)|\nabla (u+w)|^q dz \\ & \hskip 1.7cm - \int_\Omega a(z)|u+w|^{1-\gamma} dz - \lambda t^{r+\gamma-1}\|u+w\|_r^r \quad \mbox{for all $w \in W^{1,p}_0(\Omega)$}.
	\end{align*}
	We have
	\begin{align*}
		& E(0,1)=0 \quad \mbox{(since $u \in N_\lambda^+ \subseteq N_\lambda$)},\\ & E^\prime_t(0,1)=(p+\gamma-1)\|\nabla u \|_p^p + (q+\gamma-1) \int_\Omega \xi(z)|\nabla u|^q dz\\ & \hskip 1.7cm - \lambda (r+\gamma-1)\|u\|_r^r >0  \quad \mbox{(since $u \in N_\lambda^+$).}	
	\end{align*}
	
	Invoking the implicit function theorem (see Gasi\'{n}ski-Papageorgiou \cite{Ref7}, p. 481), we can find $\varepsilon>0$ and  continuous 	$\beta:B_\varepsilon(0) \to \mathbb{R}_+=(0,+\infty)$ such that		
	$$\beta(0)=1, \, \beta(w)(u+w) \in N_\lambda \quad \mbox{for all $w \in B_\varepsilon(0)$.}$$
	
	Taking $\varepsilon >0$ even smaller if necessary, we can also have
	$$ \beta(w)(u+w) \in N_\lambda^+ \quad \mbox{for all $w \in B_\varepsilon(0)$.}$$
\end{proof}	

\begin{prop}
	\label{P6} If hypotheses $H(\xi)$, $H(a)$ hold, $\lambda \in (0, \widehat{\lambda}^\ast]$, and $h \in W_0^{1,p}(\Omega)$, then we can find $b>0$ such that $\varphi_\lambda(u^\ast)\leq \varphi_\lambda(u^\ast +th)$ for all $t \in [0,b]$.
\end{prop}

\begin{proof}
	We introduce the function $\eta_h : [0,+\infty) \to \mathbb{R}$ defined by
	\begin{align} \nonumber
		\eta_h(t) & = (p-1)\|\nabla u^\ast +t \nabla h \|_p^p +(q-1) \int_\Omega \xi(z) |\nabla u^\ast + t \nabla h |^q dz \\ & \quad + \gamma \int_\Omega a(z)|u^\ast +t h|^{1-\gamma} dz - \lambda (r-1)\|u^\ast +th \|^r_r. \label{eq19}
	\end{align}		
	
	Since $u^\ast \in N_\lambda^+ \subseteq N_\lambda$ (see Proposition \ref{P4}), we have
	\begin{align}
		\label{eq20} & \gamma \int_\Omega a(z)|u^\ast|^{1-\gamma}dz= \gamma \|\nabla u^\ast \|_p^p + \gamma \int_\Omega \xi(z)|u^\ast|^{q}dz- \lambda \gamma \|u^\ast\|^r_r,\\ \label{eq21} & (p+\gamma-1)\| \nabla u^\ast \|_p^p + (q+\gamma -1) \int_\Omega \xi(z)|\nabla u^\ast |^q dz -\lambda (r+\gamma-1)\|u^\ast\|^r_r >0.
	\end{align}
	
	It follows from \eqref{eq19}, \eqref{eq20}, \eqref{eq21} that $\eta_h(0)>0$.
	
	The continuity of $\eta_h(\cdot)$ implies that we can find $b_0>0$ such that $$\eta_h(t)>0 \quad \mbox{for all $t \in [0,b_0]$.}$$
	
	On account of Lemma \ref{L5}, we can find $\vartheta(t)>0$, $t \in [0,b_0]$ such that 
	\begin{equation}
	\label{eq22} \vartheta(t)(u^\ast + th) \in N_\lambda^+, \quad \vartheta(t) \to 1 \mbox{ as $t \to 0^+$}. 
	\end{equation}
	
	Therefore we have
	\begin{align*}
		& m_\lambda^+ =\varphi_\lambda(u^\ast)\leq \varphi_\lambda(\vartheta(t)(u^\ast+th)) \quad \mbox{for all $t \in [0, b_0]$},\\ \Rightarrow \quad & m_\lambda^+ \leq \varphi_\lambda(u^\ast)\leq \varphi_\lambda(u^\ast+th) \quad \mbox{for all $t \in [0, b]$ with $0 < b \leq b_0$ (see \eqref{eq22})}.
	\end{align*}
\end{proof}	

The next proposition shows that $N^+_\lambda$ is in fact,
 a natural constraint for the energy functional $\varphi_\lambda$ (see Papageorgiou-R\v{a}dulescu-Repov\v{s} \cite{Ref11}, p. 425).

\begin{prop}
	\label{P7} If hypotheses $H(\xi)$, $H(a)$ hold and $\lambda \in (0, \widehat{\lambda}^\ast)$, then $u^\ast$ is a weak solution of \eqref{PL}.
\end{prop}

\begin{proof}
	Let $h \in W^{1,p}_0(\Omega)$ and let $b>0$ as postulated by Proposition \ref{P6}. For $0\leq t \leq b$ we have
	\begin{align*}
		& 0 \leq \varphi_\lambda(u^\ast +th)-\varphi_\lambda(u^\ast) \quad \mbox{(see Proposition \ref{P6})}, \\ \Rightarrow \quad & \frac{1}{1-\gamma}\int_\Omega a(z)\left[|u^\ast +th|^{1-\gamma}-|u^\ast|^{1-\gamma}\right]dz \\ & \leq \frac{1}{p}\left[\|\nabla u^\ast + t \nabla h\|_p^p -\|\nabla u^\ast \|_p^p\right]+\frac{1}{q}\left[\int_\Omega \xi(z) \left[|\nabla u^\ast + t \nabla h|^q -|\nabla u^\ast |^q\right]dz \right] \\ & \quad - \frac{\lambda}{r}\left[\|  u^\ast + t   h\|_r^r -\|  u^\ast \|_r^r\right].
	\end{align*}	
	
	We divide by $t>0$ and then let $t \to 0^+$. We obtain 
	\begin{align*} & \int_\Omega a(z)(u^\ast)^{-\gamma}h dz\\ &\leq \int_\Omega |\nabla u^\ast |^{p-2}(\nabla u^\ast, \nabla h)_{\mathbb{R}^N}dz +   \int_\Omega \xi(z)|\nabla u^\ast |^{q-2}(\nabla u^\ast, \nabla h)_{\mathbb{R}^N}dz \\ & \quad  - \lambda \int_\Omega  (u^\ast)^{r-1}h dz.
	\end{align*}
	
	Since $h \in W_0^{1,p}(\Omega)$ is arbitrary, equality must hold and so $u^\ast$ is a weak solution of \eqref{PL}, $\lambda \in (0, \widehat{\lambda}^\ast)$.
\end{proof}	

This proposition leads to the first positive solution of \eqref{PL}, $\lambda \in (0, \widehat{\lambda}^\ast)$.

\begin{prop}
	\label{P8} If hypotheses $H(\xi)$, $H(a)$ hold and $\lambda \in (0, \widehat{\lambda}^\ast)$, then problem \eqref{PL} admits a positive solution $u^\ast \in W^{1,p}_0(\Omega)$ such that $\varphi_\lambda(u^\ast)<0$ and $u^\ast(z)\geq 0$ for a.a. $z \in \Omega$, $u^\ast \not \equiv 0$.
\end{prop}

Next, using the set $N_\lambda^-$, we will generate a second positive solution for problem \eqref{PL} when $\lambda>0$ is small.

\begin{prop}
	\label{P9} If hypotheses $H(\xi)$, $H(a)$ hold, then we can find  $\widehat{\lambda}^\ast_0 \in (0, \widehat{\lambda}^\ast]$ such that $\varphi_\lambda \Big|_{N_\lambda^-} \geq 0$ for all $\lambda \in (0, \widehat{\lambda}^\ast_0]$.
\end{prop}
\begin{proof}
	Let $u \in N_\lambda^-$. We have
	\begin{align}
		\nonumber & (p+\gamma-1)\|\nabla u\|_p^p + (q+\gamma-1)\int_\Omega \xi(z)|\nabla u|^q dz < \lambda (r+\gamma -1)\|u\|_r^r, \\ \nonumber \Rightarrow \quad & (p+\gamma-1)c_8\| u\|_r^p < \lambda (r+\gamma -1)\|u\|_r^r \quad \mbox{for some $c_8>0$}\\ \nonumber & \hskip 1cm \mbox{(here we have used the fact that $W^{1,p}_0(\Omega) \hookrightarrow L^r(\Omega)$),}\\ \label{eq23} \Rightarrow \quad & \left[\frac{(p+\gamma-1)c_8}{\lambda(r+\gamma-1)}\right]^{\frac{1}{r-p}} \leq \|u\|_r.
	\end{align}	
	
	Arguing by contradiction, suppose that the proposition is not true. Then we can find  $u \in N_\lambda^-$ such that 
	\begin{align}
		\nonumber & \varphi_\lambda(u)<0, \\ \Rightarrow \quad & \label{eq24} \frac{1}{p}\| \nabla u \|_p^p + \frac{1}{q}\int_\Omega \xi(z)|\nabla u |^q dz - \frac{1}{1-\gamma}\int_\Omega a(z)|  u |^{1-\gamma} dz - \frac{\lambda}{r}\|u\|_r^r
		<0.
	\end{align}	
	
	We know that $u \in N_\lambda$. Therefore
	\begin{equation}
	\label{eq25} \|\nabla u\|_p^p = \int_\Omega a(z)|  u |^{1-\gamma} dz  +\lambda \|u\|_r^r -\int_\Omega \xi(z)|\nabla u |^q dz.
	\end{equation}
	
	We use \eqref{eq25} in \eqref{eq24} and obtain 
	\begin{align}\nonumber
		&\left[\frac{1}{p} -\frac{1}{1-\gamma}\right] \int_\Omega a(z)| u|^{1-\gamma} dz +\left[\frac{1}{q} -\frac{1}{p}\right] \int_\Omega \xi(z)| \nabla u|^q dz\\ \nonumber & + \lambda \left[\frac{1}{p} -\frac{1}{r}\right] \|u\|_r^r<0,\\ \nonumber \Rightarrow \, & \lambda \left[\frac{1}{p} -\frac{1}{r}\right] \|u\|_r^r< \frac{p+\gamma-1}{p(1-\gamma)}c_9 \|u\|^{1-\gamma}_r \, \mbox{for some $c_9>0$ (recall that $q<p<r$),}
		\\ \nonumber \Rightarrow \, & \|u\|_r^{r+\gamma-1}\leq \frac{(p+\gamma-1)rc_9}{\lambda (r-p)(1-\gamma)},\\ \Rightarrow \, & \label{eq26} \|u\|_r \leq 
		c_{10}\left(\frac{1}{\lambda}\right)^{\frac{1}{r+\gamma-1}}\quad \mbox{for some $c_{10}>0$}.	\end{align}
	
	We use \eqref{eq26} in \eqref{eq23} and obtain
	\begin{align*} &
		c_{11}\left(\frac{1}{\lambda}\right)^{\frac{1}{r-p}}	 \leq c_{10}\left(\frac{1}{\lambda}\right)^{\frac{1}{r+\gamma -1}} \quad \mbox{with } c_{11}=\left[\frac{(p+\gamma-1)c_8}{r+\gamma-1}\right]^{\frac{1}{r-p}}>0,\\ \Rightarrow \quad & c_{12} \leq \lambda^{\frac{1}{r-p}-\frac{1}{r+\gamma -1}}\quad \mbox{with } c_{12}=\frac{c_{11}}{c_{10}}>0,\\ \Rightarrow \quad & c_{12} \leq \lambda^{\frac{p+\gamma-1}{(r-p)(r+\gamma -1)}} \to 0 \quad \mbox{as $\lambda \to 0^+$ (since $1<p<r$, $\gamma \in (0,1)$)},
	\end{align*}
	a contradiction. Therefore we conclude that we can find $\widehat{\lambda}^\ast_0 \in (0, \widehat{\lambda}^\ast]$ such that $\varphi_\lambda \Big|_{N_\lambda^-} \geq 0$ for all $\lambda \in (0, \widehat{\lambda}^\ast_0]$.
\end{proof}	

\begin{prop}
	\label{P10} If hypotheses $H(\xi)$, $H(a)$ hold and $\lambda \in (0, \widehat{\lambda}^\ast_0]$, then there exists $v^\ast \in N^-_\lambda$, $v^\ast \geq 0$ such that $m_\lambda^-=\inf_{N_\lambda^-}\varphi_\lambda=\varphi_\lambda(v^\ast)$.
\end{prop}
\begin{proof}
	The reasoning is similar to that in the proof of Proposition \ref{P4}. If $\{v_n\}_{n \geq 1} \subseteq N_\lambda^-$  is  a minimizing sequence, then on account of Proposition \ref{P1}, we have that $\{v_n\}_{n \geq 1} \subseteq W_0^{1,p}(\Omega)$ is bounded. So, we may assume that
	$$v_n \xrightarrow{w} v^\ast \mbox{ in $W_0^{1,p}(\Omega)$ and $v_n \to v^\ast$ in $L^r(\Omega)$ as } n \to +\infty.$$ 
	
	From the proof of Proposition \ref{P4} we can find $t_0<t_2$ such that  \begin{equation}
	\label{eq27}  \mbox{$w_{v^\ast}^\prime(t_2)<0$ and $w_{v^\ast}(t_2)=\lambda \|v^\ast\|_r^r$ (see \eqref{eq11}),}
	\end{equation} 
	($t_0>0$ being the maximizer of $w_{v^\ast}$). We argue as in the proof of Proposition \ref{P4} and using \eqref{eq27}, 
	we obtain that $v^\ast \in N^-_\lambda$, $v^\ast \geq 0$, $m_\lambda^-=\varphi_\lambda(v^\ast)$.
\end{proof}

Using Lemma \ref{L5} and reasoning as in the proofs of Propositions \ref{P6} and \ref{P7}, we can also prove the following proposition. 

\begin{prop}
	\label{P11} If hypotheses $H(\xi)$, $H(a)$ hold and $\lambda \in (0, \widehat{\lambda}^\ast)$, then $v^\ast$ is a weak solution of \eqref{PL}. \qed
\end{prop}

This also completes the proof of our main result, Theorem~\ref{MT}. \qed 

\begin{rem}
	It would be interesting to study if one can get such a multiplicity result for double phase problems with a differential operator of unbalanced growth, that is, of the form 
	$$- \mbox{div}\, \left( \xi(z)|\nabla u|^{p-2}\nabla u \right)	- \Delta_q u \quad \mbox{with } 1<q<p.$$
	
	For this operator,
	 the integrand in the corresponding energy functional is
	$$k(z,t)=\frac{1}{p}\xi(z)t^p + \frac{1}{q}t^q \quad \mbox{for all }t>0.$$
	
	Note that for this integrand we have 
	$$\frac{1}{q}t^q \leq k(z,t) \leq \widehat{c}[1+t^p] \quad \mbox{for some $\widehat{c}>0$, all }t>0,$$
	(unbalanced growth). For such problems we need to work with Musielak-Orlicz-Sobolev spaces. Also, we need to strengthen the condition of $\xi(\cdot)$ ($\xi: \overline{\Omega}\to \mathbb{R}$ is Lipschitz continuous, $\xi(z)>0$ for all $z \in \Omega$), as well as restrict the exponents $1<q<p$ and require that $\frac{p}{q}<1+\frac{1}{N}$, which means that $p$ and $q$ cannot differ much (see \cite{Ref4, Ref10}).
\end{rem}

\section*{Acknowledgements}
The second author was supported by the Slovenian Research Agency grants
P1-0292, J1-8131, N1-0114, N1-0083, and N1-0064. We thank the referee for comments.

\end{document}